\documentclass[12pt,letterpaper]{amsart}
\usepackage{geometry,xcolor,graphicx}
\usepackage{amsmath,amssymb,amsfonts,amsthm,amscd}
\usepackage{mathrsfs,latexsym,url,setspace}
\usepackage[colorlinks,allcolors=blue]{hyperref} 
\usepackage{xpatch}
\xpatchcmd{\proof}
{\itshape}
{\bfseries}
{}
{}
\theoremstyle{plain}
\newtheorem{theorem}{Theorem}
\newtheorem{proposition}{Proposition}
\newtheorem{lemma}{Lemma}
\newtheorem{corollary}{Corollary}
\newtheorem{definition}{Definition}
\theoremstyle{remark}
\newtheorem*{remark}{Remark}
\newtheorem{example}{Example}

\begin{document}

\title[Projected composition operators]{Projected composition operators on pseudoconvex domains}
\author{ \v{Z}eljko \v{C}u\v{c}kovi\'c }
\begin{abstract}
Let $\Omega\subset \mathbb{C}^n$ be a smooth bounded pseudoconvex domain 
and $A^2 (\Omega)$ denote its Bergman space. Let $P:L^2(\Omega)\longrightarrow A^2(\Omega)$ 
be the Bergman projection.	For a measurable $\varphi:\Omega\longrightarrow \Omega$, 
the projected composition operator is defined by $(K_\varphi f)(z) = P(f \circ \varphi)(z), 
z \in\Omega, f\in A^2 (\Omega).$ In 1994, Rochberg studied boundedness of $K_\varphi$ 
on the Hardy space of the unit disk and obtained different necessary or sufficient 
conditions for boundedness of $K_\varphi$.  In this paper we are interested in projected 
composition operators on Bergman spaces on pseudoconvex domains. 
We study boundedness of this operator under the smoothness assumptions on 
the symbol $\varphi$ on $\overline\Omega$.
\end{abstract}

\keywords{Projected composition operators, strongly pseudoconvex domains, $\overline\partial$-Neumann operator}
\subjclass{47B33; 32A36; 32T05;}
\address{Department of Mathematics and Statistics,  University of Toledo, Toledo, Ohio, 43606}
\email{Zeljko.Cuckovic@utoledo.edu} 
\maketitle

\section{Introduction} 

Let $\Omega\subset \mathbb{C}^n$ be a smooth bounded pseudoconvex domain 
and let $H(\Omega)$ denote the space of holomorphic functions on $\Omega$.  
Let $L^2(\Omega)$ be the standard space of square-integrable functions with respect 
to the Lebesgue measure $dV$ with the usual $L^2$ norm, denoted by $\| . \|_2$. 
The Bergman space is defined as $A^2 (\Omega) = H(\Omega)\cap L^2(\Omega)$.
The Bergman projection is the orthogonal projection operator 
$P:L^2(\Omega)\longrightarrow A^2(\Omega).$

Let $\varphi:\Omega\longrightarrow \Omega$ be a measurable function given as 
$\varphi = (\varphi_1, \ldots, \varphi_n).$  We define two operators
\[(D_\varphi f)(z) = (f \circ \varphi)(z), z \in\Omega, f\in A^2 (\Omega)\]
and 
\[(K_\varphi f)(z) = P(f \circ \varphi)(z), z \in\Omega, f\in A^2 (\Omega).\]
The operator $K_\varphi$ is called the projected composition operator on 
the Bergman space $A^2(\Omega).$  When $\varphi$ is a holomorphic self 
map of $\Omega$, then  there is no need to project $f \circ \varphi,$ hence 
$K_\varphi = C_\varphi$ is the usual composition operator. Composition 
operators have been studied extensively over the last several decades.  
In particular, for the Hardy and Bergman spaces on the open unit disk 
$\mathbb{D}$, as a consequence of the Hardy-Littlewood subordination 
principle, $C_\varphi$ is always bounded.  Of course in higher dimensions, 
this will not be the case any more, for example in the open unit ball in 
$\mathbb{C}^n,$ \cite[Section 3.5]{C-M}.

Rochberg \cite{R} studied projected composition operators on the Hardy space 
of the unit disk. He established several theorems related to the boundedness of 
$K_\varphi.$ His elegant work laid out a foundation for the study of these operators.  
A precursor for his work was a paper \cite{P-M-M}. Since then, there has not been 
much work done on this topic, we mention only a recent paper by Zhao \cite{Zhao}, 
who extended some results of Rochberg.

In this paper we are concerned with the boundedness of $K_\varphi$ on Bergman 
spaces, not just on the unit disk $\mathbb{D}$ in the complex plane, but we would 
like to find boundedness conditions for 
pseudoconvex domains in $\mathbb{C}^n$.  In the next section we prove two theorems. The first one 
is for bounded pseudoconvex domains and the symbol $\varphi$ that is holomorphic near the 
boundary. The second theorem is our main result, and we prove it for smooth bounded 
pseudoconvex domains under the assumption that the symbol satisfies 
$\overline\partial \varphi_i = 0$ on some part of the boundary that is mapped 
into the boundary, for all $i$.  In both theorems, we can express the 
boundedness of $K_\varphi$ in terms of the Carleson measure condition 
on the pull back measure. If, in addition, we assume that the domain is a smooth bounded strongly pseudoconvex domain, the second theorem leads to a more concrete condition on the boundedness of $K_\varphi$ in terms of the Bergman kernel and that is the content of the important Corollary 1.

\section{Main results}

We start this section with two initial results that are not difficult to obtain.  
The proofs are analogous to the initial results obtained for the unit disk 
$\mathbb{D}.$  We include some arguments for the sake of completeness 
of our exposition.

For the operator $D_\varphi$ the boundedness is expressed in terms of a 
Carleson measure condition.

\begin{definition}\label{T:Definition}
Let $\Omega$ be a bounded domain in $\mathbb{C}^n.$ A positive measure $\mu$ on 
$\Omega$ is a Carleson measure on 	$A^2(\Omega)$ if there exists a constant $C>0$ such 
that for all $f\in A^2 (\Omega),$
\[\int_{\Omega} |f(z)|^2 d\mu(z) \leq C\|f\|_2^2.\]
\end{definition} 

Suppose now that $\Omega$ is a bounded domain and $\varphi:\Omega\longrightarrow \Omega$ 
is measurable. We want to define the pull back measure $V_\varphi$ on $\Omega$: For a set 
$E\in \Omega$, $V_\varphi(E) = V(\varphi^{-1}(E)).$

We now state the criterion for boundedness of the operator $D_\varphi.$

\begin{proposition}\label{T:Proposition}
Let $\Omega$ be a bounded domain in $\mathbb{C}^n.$  Suppose 
$\varphi:\Omega\longrightarrow \Omega$ is measurable. Then the operator $D_\varphi$ 
is a bounded operator on $A^2(\Omega)$ if and only if $V_\varphi$ is a Carleson measure 
on $A^2(\Omega).$
\end{proposition}

\begin{proof}
The proof uses a well known formula for the composition and integration; a good reference is Theorem 2.4.18 in \cite{F}.

For $f\in A^2(\Omega),$	
\[\|D_\varphi f\|_2^2 = \int_{\Omega} |f\circ \varphi(z)|^2 dV(z) = \int_{\Omega} |f(w)|^2 dV_\varphi(w).\] 
Hence if $D_\varphi$ is bounded, there exists $C>0$ such that $\|D_\varphi f\|_2^2\leq C\|f\|_2^2$ 
for all $f\in A^2(\Omega).$  This precisely means that
\[\int_{\Omega} |f(w)|^2 dV_\varphi(w) \leq C\|f\|_2^2,\]
for all $f\in A^2(\Omega),$ which implies that $dV_\varphi$ is Carleson.  

Conversely if $dV_\varphi$ is a Carleson measure, then for $f\in A^2(\Omega),$ following 
the steps above we show that 
$\|D_\varphi f\|_2^2 = \int_{\Omega} |f(w)|^2 dV_\varphi(w) \leq C\|f\|_2^2, C>0, $ which 
shows that $D_\varphi$ is bounded on $A^2(\Omega)$. 
\end{proof}

The first simple example of a bounded $K_\varphi$ is the case when $\varphi$ is a measurable function such that $\varphi(\Omega)$ is compactly supported in $\Omega.$ 

In what follows we use the notation $f(z)\lesssim g(z)$ to denote that there
exists a constant $C$, independent of $z$ such that $f(z)\leq Cg(z)$.

\begin{lemma}\label{T:Lemma1}
Let $\Omega$ be a bounded domain in $\mathbb{C}^n$ and let $\varphi:\Omega\longrightarrow \Omega$ be measurable such that $\varphi(\Omega)$ is compactly supported in $\Omega.$ Then $K_\varphi$ is bounded on $A^2(\Omega).$ 
	\end{lemma}

\begin{proof} 
Let $f\in A^2(\Omega).$  Then 
\[\|K_\varphi f\|_2 \leq \|f\circ \varphi\|_2 = \int_{\Omega}|(f\circ \varphi)(z)|^2 dV(z) 
\lesssim \sup_{w \in \varphi(\Omega)} |f(w)|^2. \]
Since $\varphi(\Omega)$ is compact in $\Omega$, the Cauchy Integral Formula gives 
$|f(w)| \lesssim \| f\|_2$ for all $w\in \varphi(\Omega)$. This then implies that 
$\|K_\varphi f\|_2^2 \lesssim \| f\|_2$, for all $f\in A^2(\Omega).$ Now we conclude that 
$K_\varphi$ is a bounded operator on $A^2(\Omega).$
\end{proof} 

To prove our first theorem, we need to use the $\overline\partial$-Neumann techniques. 
Hence, we now introduce the $\overline\partial$-Neumann operator $N$ and the Kohn's 
formula that connects the Bergman projection $P$ and $N$.

Let $\Box=\overline\partial^*\overline\partial
+\overline\partial\overline\partial^*:L^2_{(0,1)}(\Omega)\to L^2_{(0,1)}(\Omega)$ 
be the complex Laplacian, where $\overline\partial^*$ is the Hilbert space adjoint 
of $\overline\partial:L^2_{(0,1)}(\Omega)\to L^2_{(0,2)}(\Omega)$.   When $\Omega$ 
is a bounded pseudoconvex domain in $\mathbb{C}^n$, H\"{o}rmander \cite{Hor} 
showed that $\Box$ has a bounded inverse $N$ called the 
$\overline\partial$-Neumann operator. We also mention the books \cite{Ch-Sh,St2} for a thorough study of the $\overline\partial$-Neumann operator.  The $\overline\partial$-Neumann operator is an 
important tool in several complex variables. Kohn in \cite{K} 
showed that the Bergman projection is connected to the $\overline\partial$-Neumann 
operator by the formula $P=I-\overline\partial^*N\overline\partial$.  

The next theorem sort of resembles Rochberg's Theorem 4.2.

\begin{theorem}\label{T:Theorem1}
Let $\Omega$ be a bounded pseudoconvex domain, and let 
$\varphi:\Omega\longrightarrow \Omega,$ $\varphi 
=  (\varphi_1, \ldots, \varphi_n)\in C^1(\overline\Omega).$  
Let $K= \varphi(\partial\Omega)\cap \partial\Omega$ and let 
$F= \varphi ^{-1}(K) \subset \partial\Omega.$ Assume there is an open neighborhood 
$U$ of $F$ such that for every $i=1, \ldots, n$, 
$\overline\partial \varphi_i = 0$ on $U\cap \overline\Omega.$  Then $K_\varphi - D_\varphi$ 
is bounded on $A^2 (\Omega).$  Hence $K_\varphi$ is bounded if and only if $D_\varphi$ is.
\end{theorem}

\begin{proof}
In case $K= \emptyset$, we apply Lemma \ref{T:Lemma1}. 
	
Since $U$ is an open set, it follows that $\overline\Omega - U$ is compact in $\overline\Omega$, 
hence $\varphi(\overline\Omega - U) $ is also compact in $\overline\Omega$.  Now it is not 
difficult to show that 
\[\varphi(\overline\Omega - U) \cap \partial\Omega = \emptyset.\]
Suppose on the contrary, there exists $z\in \partial\Omega$ such that 
$z=\varphi(z_1), z_1 \in\overline\Omega - U $. If $z_1 \in \Omega$, then $z=\varphi(z_1)\in\Omega$, 
by the assumption on $\varphi$, which is not possible since $z\in \partial\Omega.$ 
If $z_1 \in \partial\Omega - U$, then clearly $z_1$ is not an element of $F$.  Therefore 
$z\in \Omega$, since $\varphi$ only maps elements of the boundary set $F$ into $\partial\Omega.$  
This is a contradiction again.

Using Kohn's formula, for $f\in A^2(\Omega),$  we have 
\begin{align*}
(K_\varphi-D_\varphi)f(z) &= (P-I)(f\circ\varphi)(z)\\
&= -\overline\partial^*N\overline\partial\left((f\circ\varphi)(z)\right)\\
&= -\overline\partial^*N\left(\sum_{i=1}^n \sum_{k=1}^n 
\frac{\partial f}{\partial w_k}(\varphi(z))\frac{\partial\varphi_k}{\partial \overline z_i}(z) 
d\overline z_i\right).
\end{align*}
Since $\Omega $ is a bounded pseudoconvex domain, the operator $\overline\partial^*N$ 
is bounded by \cite{Hor}.  Thus 
\[\|(K_\varphi - D_\varphi)f\|_2\lesssim 
\sum_{i=1}^n \sum_{k=1}^n \left\|\frac{\partial f}{\partial w_k}(\varphi)
\frac{\partial\varphi_k}{\partial\overline z_i}\right\|_2.\]
Let us fix $i$ and $k$ now.  Then since $\frac{\partial\varphi_k}{\partial\overline z_i} = 0$ 
on $U\cap \Omega$,

\begin{equation*}
\begin{split}
\left\|\frac{\partial f}{\partial w_k}(\varphi)\frac{\partial\varphi_k}{\partial\overline z_i}\right\|^2_2
&= \int_{\Omega} \left|\frac{\partial f}{\partial w_k}(\varphi(z))\right|^2 
\left|\frac{\partial\varphi_k}{\partial\overline z_i}(z)\right|^2 dV(z)\\
&= \int_{\Omega -U} \left|\frac{\partial f}{\partial w_k}(\varphi(z))\right|^2 dV(z)\\
&= \int_{\varphi(\Omega -U)} \left|\frac{\partial f}{\partial w_k}(\zeta)\right|^2 dV_{\varphi}(\zeta).
\end{split}
\end{equation*}

From the discussion at the beginning of the proof, it follows that $\varphi(\overline\Omega - U)$ 
is compactly supported in $\Omega$, hence by the Cauchy integral formula for the derivatives, 
we have 
\[\int_{\varphi(\Omega -U)} \left|\frac{\partial f}{\partial w_k}(\zeta)\right|^2 dV_{\varphi}(\zeta) 
\lesssim \|f\|_2^2.\] 
This proves that 
\[\|(K_\varphi-D_\varphi)f\|_2 \lesssim \|f\|_2,\]
for all $f\in A^2(\Omega).$ In other words, $K_\varphi-D_\varphi$ 
is a bounded operator on $A^2(\Omega).$
\end{proof}

It is a natural question if we can relax the assumption in the previous theorem that 
$\overline\partial \varphi_i = 0$ in a neighborhood of $F$ for all $i=1, \ldots, n.$ Namely, 
can we assume that for all $i=1, \ldots, n,$ we have $\overline\partial \varphi_i = 0$ on $F$ only.  
This problem turns out to be more difficult.  The next theorem is the main result of our paper.  
We show that it is possible to obtain the same conclusion on the boundedness of $K_\varphi$ 
under some additional smoothness assumptions. 

A typical strategy to show boundedness of a new type of operators is to show that they 
are bounded on a dense subspace of the given space. In Rochberg's case, he used the
polynomials that are dense in the Hardy space of the unit disk $\mathbb{D}$ to show that 
$K_\varphi$ is bounded.  Since we work on smooth bounded pseudoconvex domains, to 
prove our main theorem, we will use the space $A^{\infty} (\overline\Omega)$, which is 
the subspace of functions in $C^{\infty} (\overline\Omega)$ that are holomorphic in 
$\Omega$.  According to a result of \cite{C}, for a smooth bounded pseudoconvex 
domain $\Omega$, $A^{\infty}(\overline\Omega)$ is dense in $A^2(\Omega).$

In what follows, for $z\in \Omega$, let $\delta(z)$ denote the 
distance of $z$ to $\partial\Omega.$

We now state our main theorem.

\begin{theorem}\label{T:Theorem2}
Suppose that $\Omega$ is a smooth bounded pseudoconvex domain and 
let $\varphi:\Omega\longrightarrow \Omega,$ $\varphi = (\varphi_1, \ldots, \varphi_n) 
\in C^2(\overline\Omega).$ Let $K= \varphi(\partial\Omega)\cap \partial\Omega$ and 
let $F= \varphi ^{-1}(K) \subset \partial\Omega.$ Assume there is an open neighborhood 
$U$ of $F$ such that $\varphi$ is 1-1 on $U\cap \Omega.$ Also assume that for every 
$i=1, \ldots, n$, $\overline\partial \varphi_i = 0$ on $F$ and the real Jacobian $J_\varphi$ 
satisfies $|J_\varphi(z)| \geq m >0$ on $F$, for some $m > 0.$  Then $K_\varphi - D_\varphi$ 
is bounded on $A^2 (\Omega).$  Hence $K_\varphi$ is bounded if and only if $D_\varphi$ is.		
\end{theorem}

\begin{proof}
In case $K= \emptyset$, we apply Lemma \ref{T:Lemma1}.

Suppose that $f\in A^{\infty}(\overline\Omega).$ Then
\begin{align*}
\|(K_\varphi - D_\varphi)f\|^2_2
&= \|\overline\partial^*N\overline\partial(f\circ\varphi)\|_2^2\\
&= \langle \overline\partial^*N\overline\partial(f\circ\varphi), \overline\partial^*N\overline\partial(f\circ\varphi)\rangle\\
&= \langle N\overline\partial(f\circ\varphi), \overline\partial(f\circ\varphi)\rangle\\
&\leq \|N\overline\partial(f\circ\varphi)\|_2 \|\overline\partial(f\circ\varphi)\|_2\\
& \leq C\|\overline\partial(f\circ\varphi)\|_2^2,
\end{align*}
since $N$ is bounded, for some $C>0.$

Now we focus on the norm $\|\overline\partial(f\circ\varphi)\|_2 .$ As before
\[\|\overline\partial(f\circ\varphi)\|_2 
\leq \sum_{i=1}^n \sum_{k=1}^n \left\|\frac{\partial f}{\partial w_k}(\varphi)
\frac{\partial\varphi_k}{\partial\overline z_i}\right\|_2.\]
At this point we would like to use the assumptions on the symbol $\varphi,$ namely 
that $\overline\partial \varphi_k = 0$ on $F$ for all $k= 1, \ldots,n.$ This means that for all 
$i, k= 1, \ldots,n$, we have $\frac{\partial\varphi_k}{\partial\overline z_i}(z) = 0,$ for $z\in F.$  
Each $\varphi_i$ is $C^2$-smooth near the boundary so we extend it to a smooth function
 in a neighborhood of $F.$ For a sufficiently small neighborhood $U'\subset U$ of $F$, 
which we again call $U$, we now use the projection $\pi: U \rightarrow \partial\Omega$ 
which is a smooth map such that for all $ z_0 \in \partial\Omega,$   $\pi(z_0)= z_0$ and 
$\pi^{-1}(z_0)$ is a smooth curve in $U$ that intersects $\partial\Omega$ transversally 
at $z_0$.  For $z\in U$, we have $|\pi(z) - z| = \delta(z)$.  This can be seen as follows: Let $n(\zeta)$ denote the outward unit normal vector to $\partial\Omega$ at $\zeta.$ For $\epsilon$ small, the map $(t, \zeta) \rightarrow \zeta + tn(\zeta)$ is a diffeomorphism from $(-\epsilon, \epsilon) \times \partial\Omega$ onto a neighborhood of $\partial\Omega.$ Moreover, for $\epsilon$ small enough, $\pi(z)$ is the point in $\partial\Omega$ closest to $z$ (there is a radius $r > 0$ independent of $\pi(z)$ such that we have interior balls of radius $r$ tangent to $\partial\Omega$ at $\pi(z).$

Now we can use Taylor's theorem for $\frac{\partial\varphi_k}{\partial\overline z_i}$ on open balls centered at any $z_0\in F.$ In this way, for every $i, k= 1, \ldots,n$, 
we obtain small open sets $U_{k, i}\subset U$ containing $F$ such that for every 
$z\in U_{k, i}$, we have the expansion of $\frac{\partial\varphi_k}{\partial\overline z_i}$ 
around $\pi(z)=z_0 \in F$ and keeping in mind that 
$\frac{\partial\varphi_k}{\partial\overline z_i}(z_0) = 0,$ we get for some positive constants $d_{k, j},$
\begin{align*}
\left|\frac{\partial\varphi_k}{\partial\overline z_i}(z)\right|
&\leq  \sum_{j=1}^n d_{k, j} |z-z_0| + O(|z-z_0|^2)\\
&\lesssim |\pi(z) - z|\\
&= \delta(z), z \in U_{k, i}.
\end{align*}
Since $|J_\varphi| \geq m>0$ on $F$, for some $m > 0$ by the smoothness assumption 
on $\varphi$, the Jacobian will be bounded away from zero on a neighborhood of 
$F$, call it $U_0.$  In other words $|J_\varphi|(z) \geq m_0>0$ for $z\in U_0.$

Let $U' = (\cap_{k, i=1}^n U_{k, i}) \cap U_0$,  be an open neighborhood of $F.$ 
Then for any $i, k$ we have
\begin{align*}
\left\|\frac{\partial f}{\partial w_k}(\varphi)\frac{\partial\varphi_k}{\partial\overline z_i}\right\|_2^2
=& \int_{\overline\Omega-U'} \left|\frac{\partial f}{\partial w_k}(\varphi(z))\right|^2 
\left|\frac{\partial\varphi_k}{\partial\overline z_i}(z)\right|^2 dV(z)\\
& + \int_{U'} \left|\frac{\partial f}{\partial w_k}(\varphi(z))\right|^2 
\left|\frac{\partial\varphi_k}{\partial\overline z_i}(z)\right|^2 dV(z)\\
\lesssim& \|f\|_2^2 + \int_{U'} \left|\frac{\partial f}{\partial w_k}(\varphi(z))\right|^2 \delta(z)^2 dV(z).
\end{align*}
Here in the first integral we used the same Cauchy integral estimate argument as in the 
proof of Theorem \ref{T:Theorem1} (since $\varphi(\overline\Omega-U')$ is compactly 
supported in $\Omega$) and the fact that $\varphi_k$ is smooth.  In the second integral 
we used the estimate we showed above.

To finish the estimate of the integral above, we use the assumption that the real 
Jacobian $J_\varphi$ satisfies $|J_\varphi| \geq m>0$ on $F$, for some $m > 0.$  
By the smoothness assumption on $\varphi$, clearly the Jacobian $J_\varphi$ is also 
bounded from above by a bound call it $M.$ For a fixed $p\in F,$ let $B(p, r)\subset U'$ 
for some $r>0.$  Since the boundary of $\Omega$ is smooth, there exists a local 
diffeomormism $\Psi_1$ that maps $B(p, r)\cap\overline\Omega$ into the upper 
half-space $H_n^{+} = \lbrace z\in \mathbb{R}^{2n}: \text{Im}(z_n) \geq 0 \rbrace.$ 

Also $\Psi_1$ maps $B(p, r)\cap F$ into the set 
$F'\subset \lbrace z\in \mathbb{R}^{2n}: \text{Im}(z_n) = 0 \rbrace,$  and similarly there 
exists a diffeomorphism $\Psi_2$ that maps $B(\varphi(p), r)\cap\overline\Omega$ 
into the upper half-space $H_n^{+}$,  $K$ into 
$K'\subset \lbrace z\in \mathbb{R}^{2n}: \text{Im}(z_n) = 0 \rbrace.$ 

Let $p' = \Psi_1(p)\in F'.$ Now the map $\varphi$ is lifted to the map 
$\psi = \Psi_1\circ \varphi \circ \Psi_2^{-1} = (\psi_1, \ldots,\psi_n)$  that maps a 
neighborhood $V'$ of $p'$ in $H_n^{+}$ into a neighborhood $V''$ of 
$\psi(p') \in K'$ in $H_n^{+}.$  

We know that every point $z'\in F'$ has coordinates $z'= (x_1, y_1, \ldots, x_n, 0)$ 
and $\psi$ maps $V'\cap F'$ into $V''\cap K'$, hence $\psi_n(x_1, y_1, \dots, x_n, 0) = 0$ 
for $z\in V'\cap F'.$ Now it is easy to show that $\frac{\partial\psi_n}{\partial x_i} =0, i= 1, \dots, n$ 
on $V'\cap F'$ and similarly $\frac{\partial\psi_n}{\partial y_i} =0, i= 1, \ldots, n-1$ on $V'\cap F'.$ 
Thus the Jacobian 
\[J_\psi(x_1, y_1, \ldots, x_n, 0) 
= J_\psi(x_1, y_1, \ldots y_{n-1}, x_n) \times \frac{\partial\psi_n}{\partial y_n}.\] 
Since $\Psi_1$ and $\Psi_2$ are diffeomorphisms, and $m <|J_\varphi(z)| \leq M, z\in F$ for 
$m, M>0$, clearly $m' < |J_{\psi}(z)| \leq M', z\in F'$ for some constants $m', M'>0.$ Hence 
we may conclude that $|\frac{\partial\psi_n}{\partial y_n}(z)| > m'', z\in V'\cap F',$ for some 
$m''>0.$ If we write this partial derivative using its definition, we get that 
\[\left|\frac{\psi_n(x_1, y_1,\ldots, x_n, h) - \psi_n(x_1, y_1,\ldots, x_n, 0)}{h}\right| \geq s> 0\] 
for $h>0, s>0$ on $V'.$

As mentioned above, $\psi_n(x_1, y_1, \ldots, x_n, 0) = 0$ for $z\in V'\cap F'$, we get that 
\[ |\psi_n(x_1, y_1, \ldots, x_n, h)| \geq s|h|,\] 
which shows that $\delta(z) \leq s' \delta (\psi(z))$ 
on $V'.$  Pulling it back to $\Omega$, the inequality will be preserved and we obtain that 
\[\delta(z) \leq s \delta (\varphi(z)), z\in B(p, r)\] 
and hence on the whole set $U'.$ Now the second integral above can be estimated 
\begin{align}\nonumber 
\int_{U'} \left|\frac{\partial f}{\partial w_k}(\varphi(z))\right|^2 \delta(z)^2 dV(z)
&\lesssim \int_{U'} \left|\frac{\partial f}{\partial w_k}(\varphi(z))\right|^2 \delta(\varphi(z))^2 dV(z)\\
\label{Eqn1}&= \int_{\varphi(U')} \left|\frac{\partial f}{\partial w_k}(w)\right|^2 (\delta(w))^2 |J_{\varphi^{-1}}(w)| dV(w)\\
\nonumber &\lesssim \int_{\Omega} |\nabla f(w)|^2 (\delta(w))^2  dV(w),
\end{align}
since $J_{\varphi}$ is bounded from below on $U'$ and hence $J_{\varphi^{-1}}$ is bounded 
above there.

The last integral in \eqref{Eqn1} is less than equal to $C\|f\|_2^2$ where $C$ is independent of $f$; that was proved for Lipschitz 
domains by Boas and Straube \cite[Proposition 2.1]{Bo-St}.  As for published references, for a domain with a $C^1$ boundary, Theorem 1 in \cite{Det} essentially shows that the last integral in \eqref{Eqn1} is dominated (modulo constants) by $\|f\|_2^2$. We would also like to point out that the last 
integral is in fact equivalent  (modulo constants) to $\|f\|_2^2$ for starshaped domains, the result proved 
by Straube \cite{St1}. Since the smooth domains are locally starshaped, the general case can be deduced from this special case.

Then for any $i, k$ from above, we have
\[\left\|\frac{\partial f}{\partial w_k}(\varphi)\frac{\partial\varphi_k}{\partial\overline z_i}\right\|_2 
\lesssim \|f\|_2.\]
Now the norm
\begin{align*}
\|\overline\partial(f\circ\varphi)\|_2
&\leq \sum_{i=1}^n \sum_{k=1}^n \left\|\frac{\partial f}{\partial w_k}(\varphi)\frac{\partial\varphi_k}{\partial\overline z_i}\right\|_2\\
&\lesssim \|f\|_2,
\end{align*}
which shows
\[\|(K_\varphi - D_\varphi)f\|^2_2 \leq \|\overline\partial(f\circ\varphi)\|_2^2 \lesssim \|f\|_2^2,\] 
for $f\in A^{\infty}(\overline\Omega).$ By the density of $A^{\infty}(\overline\Omega)$ in 
$A^2(\Omega)$ we reach the conclusion that $K_\varphi - D_\varphi$ is bounded on 
$A^2(\Omega)$ and that finishes the proof.
\end{proof}

\begin{remark}
In the proof above we showed that $(K_\varphi - D_\varphi)f = - \overline\partial^*N(Tf)$, where $T$ is continuous in the $L^2$-norm.  Because $\overline\partial^*N$ is continuous on $L^2(\Omega),$ $K_\varphi - D_\varphi$ is.  However if the pseudoconvex domain $\Omega$ has the property that $N$ (and hence $\overline\partial^*N$) is compact, we conclude that $K_\varphi - D_\varphi$ is compact.  Therefore $K_\varphi$ and $D_\varphi$ are simultaneously compact or simultaneously Fredholm.  A good reference for when $N$ is compact is \cite{St2}.

\end{remark}

The following corollary gives us a concrete criterion for boundedness of $K_\varphi$ under the assumptions of Theorem 2 and the additional condition that $\Omega$ is strongly pseudoconvex.

Suppose that $\Omega$ is smooth bounded, that is: there exists a $C^\infty$,
real-valued function $r:\,\,\text{nbhd}
\left(\overline\Omega\right)\longrightarrow\mathbb{R}$ such that
$\Omega=\{ z: r(z) <0\}$ and $dr\neq 0$ when $r=0$. The domain $\Omega$ is strongly 
pseudoconvex, if the complex Hessian $i\partial\bar\partial r(p)\left(\xi,\bar\xi\right) >0$ 
for all $p\in b\Omega$
and all vectors $\xi\in\mathbb{C}^n$ 
satisfying $\partial r(p)\left(\xi\right)=0$.  Near the boundary $r(z)$ and 
$\delta(z)$ are equivalent.

Recall now Proposition \ref{T:Proposition} which characterizes the boundedness of $D_\varphi$ in 
terms of the pull back measure $V_\varphi$ being Carleson for $A^2(\Omega).$  
Under the hypotheses of Theorem \ref{T:Theorem2}, we can conclude that $K_\varphi$ is bounded 
if and only if the measure $V_\varphi$ is Carleson for $A^2(\Omega).$  However Carleson 
measures for strongly pseudoconvex domains have been characterized before in the 
paper \cite{C-Zh}.  Later Abate, Raissy and Saracco gave a complete characterization of 
these measures for the weighted Bergman spaces on strongly pseudoconvex 
domains \cite{A-R-S}. As expected the characterizaion of Carleson measures on 
$A^2(\Omega)$ is expressed in terms of the Bergman kernel $K(z, w)$.  
Let $k_z(w) = K(z, z)^{-\frac{1}{2}} K(z, w)$ denote the normalized Bergman kernel. 
The following result is a special case of \cite[Theorem 2.2]{C-Zh}:

\begin{theorem}[\v{C}u\v{c}kovi\'c \& Zhao]\label{TheoremA} 
Assume that $\Omega\subset \mathbb{C}^n$ is a smooth bounded strongly pseudoconvex 
domain. A positive measure $\mu$ on $\Omega$ is a Carleson measure for $A^2(\Omega)$ 
if and only if
\[\sup_{z\in\Omega} \int_{\Omega} |k_z(w)|^2 d\mu(w) < \infty.\]
\end{theorem}

Now the following corollary follows immediately.

\begin{corollary}\label{T:Corollary 3}
Let $\Omega\subset \mathbb{C}^n$ be a smooth bounded strongly pseudoconvex 
domain. Then under the assumptions of Theorem \ref{T:Theorem2}, the operator 
$K_\varphi$ is bounded if and only if 
\[\sup_{z\in\Omega} \int_{\Omega} |k_z(\varphi(w))|^2 dV(w) < \infty.\]
\end{corollary}

At this point we would like to point out that the condition in the corollary above is also necessary 
and sufficient for boundedness of an arbitrary composition operator $C_\varphi$ acting on the
 Bergman space of a smooth bounded strongly pseudoconvex domain.  This result is a special 
case of \cite[Theorem 2.1]{C-Zh}.

We would like to finish the paper with an example in dimension $n=1.$  
A variation of this example could be constructed for the open unit ball in $\mathbb{C}^n.$

\begin{example} 
Let $0 \leq \chi(z) \leq 1$ be a smooth cut-off function defined in a neighborhood 
of $\overline{\mathbb{D}}$ defined so that $\chi(z) = 1$ for $\frac {3}{4} < |z| < 1 + \epsilon, \epsilon >0$ 
and $\chi(z) = 0$ for $|z| \leq \frac{1}{2}.$

Define the symbol $\varphi(z) = z(2|z| - |z|^2) \chi(z),$ which is a smooth function on a 
neighborhood of $\mathbb{D}.$ Then clearly $\varphi(z) = 0$ for $|z| \leq \frac{1}{2}.$

For $|z|=1, \varphi(z) = z,$ hence $\varphi$ maps $\partial\mathbb{D}$ into itself and 
$\varphi : \mathbb{D} \rightarrow \mathbb{D}$. Thus in our case $F = K= \partial \mathbb{D}$.

We can write $\varphi(z) = u(x, y) + iv(x, y)$, where $u(x, y) = x(\sqrt{x^2+y^2} - (x^2+y^2))$ 
and $v(x, y) = y(\sqrt{x^2+y^2} - (x^2+y^2)).$  Then one can verify that 
$\overline \partial \varphi=0$ on $F$ and $J_\varphi(z)=1, z\in F.$

Suppose $z, w\in \lbrace z: \frac {3}{4} < |z| \leq 1 \rbrace, z = re^{i\theta}, w = \rho^{i\psi}$
 and $\varphi(z)=\varphi(w).$  Then we have $r^2(2-r) e^{i\theta} = {\rho}^2(2-\rho) e^{i\psi}$
 which shows that $r^2(2-r) = {\rho}^2(2-\rho)$ and $\theta = \psi + 2\pi$.  Using the fact 
that $r\rightarrow r^2(2-r)$ is an increasing function on $(0, 1)$, we conclude that $z=w.$ 
Hence $\varphi$ is a 1-1 function on $\lbrace z: \frac {3}{4} < |z| \leq 1 \rbrace.$

We are interested in the question of boundedness of $K_\varphi$ acting on the 
Bergman space $A^2(\mathbb{D}),$ where the measure is the normalized Lebesgue 
area measure $dA.$  Recall that the Bergman kernel on the unit disk has the form 
$K_z(w) = \frac{1}{(1-\overline z w)^2}$ and hence $k_z(w) = \frac{1-|z|^2}{(1-\overline z w)^2}.$  

According to the Corollary \ref{T:Corollary 3}, $K_\varphi$ is bounded if and only if 
\[\sup_{z\in\Omega} \int_{\Omega} |k_z(\varphi(w))|^2 dA(w) < \infty,\] 
which in our case means
\[\sup_{z\in \mathbb{D}} (1-|z|^2)^2 \int_{\mathbb{D}} 
\frac{1}{|1-z \overline \varphi(w)|^4} dA(w) < \infty.\]
Hence we need to show that
\begin{align*}
\sup_{z\in\mathbb{D}} (1-|z|^2)^2 \left(\int_0^{1/2} \int_0^{2\pi} 1 dA +  \int_{1/2}^1 
\int_0^{2\pi}\frac{1}{|1-z \rho(2\rho - \rho^2) \chi(\rho e^{i\theta}) 
e^{-i\theta}|^4} d\theta \rho d\rho\right) < \infty.
\end{align*}
Thus we only have to focus on the second integral. To estimate the inside integral 
we use the Forelli-Rudin integral estimate for the unit circle in dimension $n=1:$
\[\int_{\partial\mathbb{D}} \frac{1}{|1 - \langle z, \zeta\rangle |^{1+c}} d\zeta \approx (1-|z|^2)^{-c}.\]
In our case $c=3$ hence we can estimate the inside integral and get
\begin{align*}
\int_{1/2}^1 \int_0^{2\pi} \frac{1}{|1-z \rho(2\rho - \rho^2) 
\chi(\rho e^{i\theta}) e^{-i\theta}|^4} d\theta \rho d\rho 
\leq& \int_{1/2}^1 \frac{1}{|1-|z|^2 \rho^2(2\rho-\rho^2)^2|^3} \rho d\rho\\	
\leq& \int_{1/2}^1 \frac{1}{(1-|z|^2 \rho^2)^3} \rho d\rho\\  
=& \int_{1/2}^1 \frac{1}{(1-|z| \rho)^3 (1 + |z| \rho)^3} \rho d\rho\\
 \leq & \int_{1/2}^1 \frac{1}{(1-|z| \rho)^3} d\rho\\
\left(\text{using } u = 1-|z| \rho\right) \quad \quad  
=& \frac{1}{|z|}\left(\frac{1}{(1-|z|)^2} - \frac{1}{(1-\frac{1}{2}|z|)^2}\right)\\
 =& \frac{1-\frac{3}{4}|z|}{(1-|z|)^2 (1-\frac{1}{2}|z|)^2}.
\end{align*}
Now we have 
\begin{align*}
\sup_{z\in \mathbb{D}} (1-|z|^2)^2 \left(\int_{1/2}^1 \int_0^{2\pi}
\frac{1}{|1-z \rho(2\rho - \rho^2) \chi(\rho e^{i\theta}) e^{-i\theta}|^4} d\theta \rho d\rho\right) 
& \leq \sup_{z\in\mathbb{D}} \frac{(1-|z|^2)^2}{(1-|z|)^2} \frac{1-\frac{3}{4}|z|}{(1-\frac{1}{2}|z|)^2}\\
& \leq \sup_{z\in\mathbb{D}} (1+|z|)^2\left(\frac{1-\frac{3}{4}|z|}{(1-\frac{1}{2}|z|)^2}\right)\\
& < \infty.
\end{align*}
Corollary \ref{T:Corollary 3} now shows that $K_\varphi$ is bounded on $A^2(\mathbb{D}).$
\end{example}

\section{Acknowledgment} 
The author would like to thank the referee for reading the manuscript carefully and for the suggestions that improved the paper.  We also thank S\"onmez \c{S}ahuto\u{g}lu for several illuminating 
discussions about the content of this paper, 
Emil Straube for a helpful conversation and for sending us his unpublished manuscript written 
with Harold Boas.  Finally we thank Trieu Le for pointing out a mistake in a previous version of 
the paper.


\begin{thebibliography}{99}
\bibitem[A-R-S]{A-R-S}
M. Abate, J. Raissy \& A. Saracco,
\emph{Toeplitz operators and Carleson measures in strongly pseudoconvex domains},
J. Funct. Anal. \textbf{263} (2012), 3449-3491

\bibitem[Bo-St]{Bo-St}
H. P. Boas \& E. J. Straube,
\emph{Sobolev norms of harmonic and analytic functions},
Unpublished

\bibitem[C]{C}
D. Catlin,
\emph{Boundary behavior of holomorphic functions on pseudoconvex domains},
J. Differential Geom.
\textbf{15} (1980), 605-625

\bibitem[Ch-Sh]{Ch-Sh}
S. Chen \& M. Shaw,
\emph{Partial differential equations in several complex variables},
Studies in Adv. Math. \textbf{19} AMS/IP (2001)

\bibitem[C-M]{C-M}
C. Cowen \& B. MacCluer,
\emph{Composition operators on spaces of analytic functions},
CRC Press, Studies in Advanced Mathematics (1995)

\bibitem[C-Zh]{C-Zh}
 \v{Z}. \v{C}u\v{c}kovi\'c \& R. Zhao,
 \emph{Essential norm estimates of weighted composition operators between Bergman spaces on strongly pseudoconvex domains},
Math. Proc. Camb. Phil. Soc. \textbf{142} (2007), 525-533

\bibitem[D]{Det}
J. Detraz,
\emph{Classes de Bergman de fonctions harmoniques},
Bull. de la S.M.F. \textbf{109} (1981), 259-268

\bibitem[F]{F}
H. Federer,
\emph{Geometric measure theory},
Springer, Classics in Mathematics, Berlin (1969)

\bibitem[Hor]{Hor}
L. H\"ormander,
\emph{$L^2$ estimates and existence theorems for the
$\overline\partial$-operator},
Acta Math. \textbf{113} (1965), 89-152
 
\bibitem[K]{K}
J. J. Kohn,
\emph{Harmonic integrals on strongly pseudo-convex manifolds I}, Annals of Math, \textbf{78} (1963), 112-148 


 \bibitem[P-M-M]{P-M-M}
 S. Pattanayak, C. K. Mohapatra, \& A. K. Mishra,
 \emph{A new class of composition operators},
 Internat. J. Math. Math. Sci. \textbf{10}  (1987),  473-482

\bibitem[R]{R}
R. Rochberg,
\emph{Projected composition operators on the Hardy space},
Indiana University Math. J. \textbf{43} (1994),  441-458

\bibitem[St1]{St1}
E. J. Straube
\emph{Interpolation between Sobolev and between Lipschitz spaces of analytic functions on starshaped domains}, Trans. Amer. Math. Soc. \textbf{316} (1989), 653-671


\bibitem[St2]{St2}
E. J. Straube
\emph{Lectures on the $L^2$-Sobolev theory and of the $\overline\partial$-Neumann Problem}, ESI Lectures in Mathematics and Physics, Eur. Math. Soc. Zurich, 2010

\bibitem[Zhao]{Zhao}
 C. Zhao,
 \emph{Boundedness of projected composition operators over the unit disc},
 J. Math. Anal. Appl. \textbf{467} (2018),  521-536

\end{thebibliography}
\end{document}